\documentclass{amsart}
\usepackage{amsmath, amsthm, amscd, amssymb,latexsym,enumerate}
\usepackage{url}
\usepackage{array}
\usepackage[all]{xy}
\usepackage{enumitem}
\usepackage{hyperref}

\def\id{\mathrm{id}}
\def\prid{\mathrm{prid}}
\def\Z{\mathbb{Z}}
\def\Q{\mathbb{Q}}
\def\F{\mathbb{F}}

\def\C{\mathbb{C}}

\def\SS{{\mathcal S}}

\def\BB{{\mathcal B}}

\def\RR{{\mathcal R}}

\def\G{{G}}

\def\Spec{\mathrm{Spec}}

\def\i{\mathrm{i}}

\def\sep{\mathrm{sep}}

\def\image{\mathrm{image}}
\def\proj{\mathrm{proj}}

\def\im{\mathrm{im}}
\def\Hom{\mathrm{Hom}}

\def\dim{\mathrm{dim}}

\def\log{\mathrm{log}}

\def\rk{\mathrm{rank}}

\def\order{\mathrm{order}}

\def\m{{\mathfrak m}}
\def\mm{{\mathfrak m}}
\def\ff{{\mathfrak f}}

\def\nn{{\mathfrak n}}
\def\a{{\mathfrak a}}
\def\b{{\mathfrak b}}

\def\F{{\mathbb F}}

\def\onto{\twoheadrightarrow}

\def\isom{\xrightarrow{\sim}}

\numberwithin{equation}{section}

\newtheorem{thm}{Theorem}
\numberwithin{thm}{section}

\newtheorem{lem}[thm]{Lemma}
\newtheorem{cor}[thm]{Corollary}
\newtheorem{prop}[thm]{Proposition}

\theoremstyle{definition}
\newtheorem{defn}[thm]{Definition}
\newtheorem{notation}[thm]{Notation}
\newtheorem{algorithm}[thm]{Algorithm}

\newtheorem{ex}[thm]{Example}
\newtheorem{exs}[thm]{Examples}
\newtheorem{rem}[thm]{Remark}

\title
[Roots of unity in orders]
{Roots of unity in orders}
\author[H.\ W.\ Lenstra, Jr.]{H.\ W.\ Lenstra, Jr.}
\address{Mathematisch Instituut, Universiteit Leiden, The Netherlands}
\email{hwl@math.leidenuniv.nl}
\author[A.\ Silverberg]{A.\ Silverberg}
\address{Department of Mathematics, University of California, Irvine, CA 92697}
\email{asilverb@math.uci.edu}
\thanks{This material is based on research sponsored by DARPA under agreement number FA8750-13-2-0054 and by the Alfred P.~Sloan Foundation. The U.S.\ Government is authorized to reproduce and distribute reprints for Governmental purposes notwithstanding any copyright notation thereon. The views and conclusions contained herein are those of the authors and should not be interpreted as necessarily representing the official policies or endorsements, either expressed or implied, of DARPA or the U.S.\ Government.
}

\dedicatory{Communicated by John Cremona}

\begin{document}

\begin{abstract} 
We give deterministic polynomial-time algorithms
that, given an order, compute the primitive idempotents and
determine a set of generators for the group of roots of unity
in the order.
Also, we show that the discrete logarithm problem in the group of roots of unity
can be solved in polynomial time.
As an auxiliary result, we solve the discrete logarithm problem for certain unit
groups in finite rings.
Our techniques, which are taken from commutative algebra, may have
further potential in the context of cryptology and computer algebra.
\end{abstract}

\maketitle

\section{Introduction}
An {\em order} is a commutative ring whose additive group
is isomorphic to $\Z^n$ for some non-negative integer $n$.
The present paper contains algorithms for computing the
{\em idempotents} and the {\em  roots of unity} of a given order.

In algorithms, we specify an order $A$ by listing a system
of ``structure constants'' $a_{ijk}\in\Z$ with $i,j,k\in\{ 1,2,\ldots,n\}$;
these determine the multiplication in $A$ in the sense that for some
$\Z$-basis $e_1,e_2,\ldots,e_n$ of the additive group of $A$, 
one has $e_ie_j = \sum_{k=1}^n a_{ijk}e_k$ for all $i,j$.
The elements of $A$ are then represented by their coordinates
with respect to that basis.

An idempotent of a commutative ring $R$ is an element
$e\in R$ with $e^2=e$, and we denote by
$\id(R)$ the set of idempotents.
An idempotent $e\in\id(R)$ is called {\em primitive} if 
$e\neq 0$ and for all $e'\in\id(R)$ one has $ee'\in\{0,e\}$;
let $\prid(R)$ denote the set of primitive idempotents of $R$.

Orders $A$ have only finitely many idempotents, but they
may have more than can be listed by a
polynomial-time algorithm; however, if one knows
$\prid(A)$, then one implicitly knows $\id(A)$, since there
is a bijection from the set of subsets of $\prid(A)$ to $\id(A)$
that sends $W \subset \prid(A)$ to $e_W=\sum_{e\in W}e \in\id(A)$.
For $\prid(A)$ we have the following result.

\begin{thm}
\label{idempintrothm}
There is a deterministic polynomial-time algorithm (Algorithm \ref{idempotentalg})
that, given an order $A$, lists all primitive idempotents of $A$.
\end{thm}

A {\em root of unity} in a commutative ring $R$ is an element
of finite order of the group $R^\ast$ of invertible elements of $R$;
we write
$\mu(R)$ for the set of roots of unity in $R$, which is a subgroup of $R^\ast$.

As with idempotents, orders $A$ have only finitely many  roots of unity,
but possibly more than
can be listed by a polynomial-time algorithm, and to
control $\mu(A)$ we shall use generators and relations.
If $S$ is a finite system of generators for an abelian group $G$,
then by a {\em set of defining relations} for $S$ we mean a system of
generators for the kernel of the surjective group homomorphism
$\Z^S \to G$, $(m_s)_{s\in S} \mapsto \prod_{s\in S} s^{m_s}$.

\begin{thm}
\label{muAthm}
There is a deterministic polynomial-time algorithm (Algorithm \ref{rootsofunityrelsalg})
that, given an order $A$, produces a set $S$ of generators of $\mu(A)$,
as well as a set of defining relations for $S$. 
\end{thm}

Theorem \ref{muAthm}, which provides a key ingredient in an algorithm
for lattices with symmetry that was recently developed by the authors \cite{LenSil,LwS}, is our main result,
and its proof occupies most of the paper.
It makes use of
several techniques from commutative algebra that so far have found little 
employment in an algorithmic context.
A sketch appeared in Proposition 4.7 of \cite{LenSil}.

We shall also obtain a solution to the discrete logarithm
problem in $\mu(A)$ and all its subgroups, and more generally in
all subgroups of the group $\mu(A\otimes_\Z\Q)$, which is still finite.
Note that $A\otimes_\Z\Q$ is a ring containing $A$ as a subring,
and that a $\Z$-basis for $A$ is a $\Q$-basis for the additive group
of $A\otimes_\Z\Q$. If one replaces $\mu(A)$ by $\mu(A\otimes_\Z\Q)$
in Theorem \ref{muAthm}, then it remains true, and in fact
it becomes much easier to prove (Proposition \ref{muEthmalgorworks}).
Our  solution to the discrete logarithm
problem in $\mu(A\otimes_\Z\Q)$ and all of its subgroups, in
particular in $\mu(A)$, reads as follows.

\begin{thm}
\label{muAthm2}
There is a deterministic polynomial-time algorithm 
that, given an order $A$, a finite system $T$ of elements of $\mu(A\otimes_\Z\Q)$,
and an element $\zeta\in A\otimes_\Z\Q$, decides whether $\zeta$ belongs to the
subgroup $\langle T\rangle \subset \mu(A\otimes_\Z\Q)$ generated by $T$, 
and if so finds $(m_t)_{t\in T}\in\Z^T$ with $\zeta = \prod_{t\in T} t^{m_t}$.
\end{thm}

We shall prove Theorem \ref{muAthm2} in section \ref{DLsect}, as a consequence of
the results on $\mu(A\otimes_\Z\Q)$ in section~\ref{Qalssect}
and a number of formal properties of ``efficient presentations''
of abelian groups that are developed in section \ref{DLsect}.

A far-reaching generalization of Theorem \ref{muAthm2}, in which
$\mu(A\otimes_\Z\Q)$ is replaced by the full unit group $(A\otimes_\Z\Q)^\ast$,
is proven in \cite{Qalgs}.

Of the many auxiliary results that we shall use, there are two that
have independent interest.
The first concerns the discrete logarithm problem in certain unit groups
of finite rings, and it reads as follows.

\begin{thm}
\label{muAthm4}
There is a deterministic polynomial-time algorithm 
that, given a finite commutative ring $R$ and a nilpotent ideal $I\subset R$,
produces a set $S$ of generators of the subgroup $1+I \subset R^\ast$,
as well as a set of defining relations for $S$. 
Also, there is a deterministic polynomial-time algorithm 
that, given $R$ and $I$ as before, as well as a finite system $T$
of elements of $1+I$ and an element $\zeta\in R$, decides whether
$\zeta$ belongs to the
subgroup $\langle T\rangle \subset 1+I$, 
and if so finds $(m_t)_{t\in T}\in\Z^T$ with $\zeta = \prod_{t\in T} t^{m_t}$.
\end{thm}

The proof of this theorem is given in section \ref{Isect}.
It depends on the resemblance of $1+I$ to the additive group $I$,
in which the discrete logarithm problem is easy.

The second result that we single out for special mention is of a 
purely theoretical nature.
Let $R$ be a commutative ring.
For the purposes of this paper, commutative rings have an identity
element $1$ (which is $0$ if and only if the ring is the $0$ ring).
We call $R$  {\em connected} if $\#\id(R) = 2$ or, equivalently,
if $\id(R) = \{ 0,1\}$ and $R\neq \{ 0\}$.
A polynomial $f\in R[X]$ is called {\em separable}
(over $R$) if $f$ and its formal derivative
$f'$ generate the unit ideal in $R[X]$.
For example, $f=X^2-X$ is separable because $(f')^2 -4f =1$.

\begin{thm}
\label{degbdintrothm}
Let $R$ be a connected commutative ring, and let $f \in R[X]$ be separable. Then
$f\neq 0$ and $\#\{ r\in R : f(r) = 0 \} \le \deg(f)$.
\end{thm}

For the elementary proof, see section \ref{sepconnsect}.

While, technically, one must admit that Theorem \ref{degbdintrothm}
plays only a modest role in the paper, it does convey an
important message, namely that zeroes of polynomials that are separable
are easier to control than zeroes of other polynomials.
Thus, $X^2-X$ is separable over any $R$, while $X^m-1$ (for $m\in\Z_{>0}$)
is separable if and only if $m\cdot 1\in R^\ast$, a condition that
for a non-zero order and $m > 1$ is never satisfied;
accordingly, Theorem \ref{idempintrothm} is much easier to prove than Theorem \ref{muAthm}.

We next provide an overview of the algorithms that underlie
Theorems \ref{idempintrothm} and \ref{muAthm}.
In both cases, one starts by reducing the problem, in a fairly
routine manner, to the special case in which each element of $A$ is
a zero of some separable polynomial in $\Q[X]$;
for the rest of the introduction we assume that the latter condition
is satisfied.
Then the $\Q$-algebra $E=A\otimes_\Z\Q$ can be written as the product
of finitely many algebraic number fields $E/\m$, with
$\m$ ranging over the finite set $\Spec(E)$ of prime ideals of $E$;
hence $\prid(E)$ is in bijection with $\Spec(E)$.
The image of $A\subset E$ under the map $E\to E/\m$
may be identified with the ring $A/(\m\cap A)$,
so that $A$ becomes a subring of the product ring
$B = \prod_{\m\in\Spec(E)}A/(\m\cap A)$;
this is also an order, and it is ``close" to $A$ in the sense
that the abelian group $B/A$ is finite.
The ring $B$ has many idempotents, in the sense that
$\id(B)$ equals all of $\id(E)$, and 
$\#\prid(B)=\#\Spec(E)$.
To determine which subsets $W\subset \prid(B)$ give rise to
idempotents that lie in $A$, we define a certain graph $\Gamma(A)$
with vertex set $\Spec(E)$ such that the connected components
of $\Gamma(A)$ correspond exactly to the primitive idempotents
of $A$. This leads to Theorem \ref{idempintrothm}.

To prove Theorem \ref{muAthm}, one likewise starts from $B$, 
generators for $\mu(B)$ being easily found by
standard algorithms from algebraic number theory.
However, there is no standard way of computing $\mu(A) = \mu(B)\cap A$,
which is the intersection of a multiplicative group and an additive group,
and we must proceed in an indirect way.
For a prime number $p$, denote by $\mu(A)_p$ the group of roots of unity
in $A$ that are of $p$-power order, and likewise $\mu(B)_p$.
Then $\mu(A)$ is generated by its subgroups $\mu(A)_p =  \mu(B)_p \cap A$,
with $p$ ranging over the set of primes dividing $\#\mu(B)$;
all these $p$ are ``small".
It will now suffice to fix $p$ and determine generators for  $\mu(A)_p$.
To this end, we introduce the intermediate order
$A \subset C \subset B$ defined by $C =  A[1/p] \cap B$.
The finite abelian group $B/C$ is of order coprime to $p$,
and it turns out that this makes it relatively easy to determine
$\mu(C)_p =  \mu(B)_p \cap C$; in fact, one of the results 
(Proposition \ref{degbd}(b))
leading up to Theorem \ref{degbdintrothm}  stated above
shows that this can be done by exploiting the graph $\Gamma(C)$
that we encountered in the context of idempotents.
The passage to $\mu(A)_p =  \mu(C)_p \cap A$ is of an
entirely different nature, as $C/A$ is of order a power of $p$.
It is here that we have to invoke Theorem \ref{muAthm4} for
certain finite rings $R$ that are of $p$-power order.

It is important to realize that the only reason that an intersection
such as $\mu(A) = \mu(B)\cap A$ is hard to compute is that
$\mu(B)$, though finite, may be {\em large}---testing
each element of $\mu(B)$ for membership in $A$ will not lead
to a polynomial-time algorithm. 
By contrast, the {\em exponent} of each group $\mu(B)_p$ is
{\em small} (Lemma \ref{muApremark}(iv)), so results stating 
that certain subgroups of $\mu(B)_p$ are cyclic---of which there are
several in the paper---are valuable in obtaining a polynomial
bound for the runtime of our algorithm.

\section{Definitions and examples}

From now on, 
when we say commutative $\Q$-algebra we will mean a 
commutative $\Q$-algebra that is finite-dimensional as a $\Q$-vector space.
See \cite{AtiyahMcD,LangAlg} for background on commutative rings and linear algebra.

\begin{defn}
If $A$ is an order whose additive group
is isomorphic to $\Z^n$, we call $n$
the {\bf rank} of $A$.
\end{defn}

If the number  
of idempotents in $R$ is finite,
then each idempotent is  the sum of a unique subset of
$\prid(R)$, 
and one has $\#\id(R) = 2^{\#\prid(R)}$.

\begin{defn}
\label{conndef}
A commutative ring $R$ is called {\bf connected} if $\#\{ x\in R : x^2=x\} = 2$.
\end{defn}

\begin{defn}
If $R$ is a commutative ring, let 
$\Spec(R)$ denote the set of prime ideals of $R$.
\end{defn}

Although we do not use it, we point out
that a commutative ring $R$ is connected if and only if $R \neq 0$ and
$R$ cannot be written as a product of 2 non-zero rings.
The definition is motivated by the fact that a commutative ring $R$ is connected
if and only if $\Spec(R)$ is connected. 
(A topological space is connected if and only if it has exactly 2
open and closed subsets.)

\begin{notation}
If $G$ is a group and $p$ is a prime number, define
$$G_p = \{ g\in G : g^{p^r} =1 \text{ for some $r\in\Z_{\ge 0}$} \}.$$
\end{notation}

\begin{defn}
Suppose $R$ is a commutative ring.
A polynomial $f\in R[X]$ is {\bf separable} over $R$ if
$$
R[X] f + R[X] f' = R[X],
$$
where if $f = \sum_{i=0}^ta_iX^i$ then 
$f' = \sum_{i=1}^t ia_iX^{i-1}$.

\end{defn}

One can show that if $f$ is a monic polynomial over a commutative ring $R$,
then $f$ is separable over $R$ if and only if its discriminant is a unit in $R$.

\begin{defn}
Suppose $E$ is a  commutative $\Q$-algebra.
If $\alpha\in E$, then $\alpha$ is {\bf separable} over $\Q$
if there exists a separable polynomial $f\in\Q[X]$ such that 
$f(\alpha)=0$.
Let $E_\sep$ denote the set of $y\in E$ that are separable over $\Q$. 
We say $E$ is separable over $\Q$ if $E_{\sep}=E$.
\end{defn}

We note that $E_\sep$ is a commutative $\Q$-algebra
(see for example Theorem 1.1 of \cite{Qalgs}).

\begin{defn}
Suppose $R$ is a commutative ring.
An element $x \in R$ is called {\bf nilpotent} if there exists $n\in\Z_{>0}$ 
such that $x^n = 0$.
An ideal $I$ of $R$ is called  nilpotent if there exists $n\in\Z_{>0}$ 
such that $I^n = 0$, where $I^n$ is the product of $I$ with itself $n$ times.
The set of nilpotent elements of $R$ is an ideal, called the 
{\bf nilradical} and denoted $\sqrt{0}$ or $\sqrt{0_R}$.
\end{defn}

\begin{exs}
The polynomial 
$X^2-X$ is separable over every ring.
A linear polynomial
$aX+b$ is separable over $R$ if and only if the 
$R$-ideal generated by $a$ and $b$ is $R$. 
If $m\in\Z_{\ge 0}$, then the 
polynomial $X^m-1$ is separable over $R$ if and only if $m\cdot 1$ is a unit
in $R$. 
\end{exs}

\begin{ex}
\label{Zxfex}
Suppose $f(X)\in\Z[X]$ is a monic polynomial of degree $n$.
Then the ring 
$\Z[X]/(f)$ is an order of rank $n$.
We remark that 
the map $e \mapsto \gcd(e,f)$
is a bijection from the set of idempotents of $\Z[X]/(f)$ to 
$\{ g\in\Z[X] : \text{ $g$ is monic, } g|f, \text{ and } R(g,f/g)=\pm 1 \},$
where $R(g,f/g)$ is the resultant of $g$ and $f/g$.
\end{ex}

\begin{ex}
If $G$ is a finite group of order $2n$ with a fixed element $u$
of order $2$, then  $\Z\langle\G\rangle = \Z[G]/(u+1)$
is a connected order of rank $n$, and $\mu(\Z\langle\G\rangle) = G$
(see Remark 16.3 of \cite{LwS}).
\end{ex}

\begin{ex}
If $n\in\Z_{> 0}$ and 
$A = \{ (a_i)_{i=1}^{n} \in\Z^n : a_i \equiv a_j \mod 2 \text{ for all $i,j$}\}$
 with componentwise addition and multiplication,
then $A$ is a connected order, 
$\mu(A) = \{ (\pm 1,\ldots,\pm 1) \}$, and $\#\mu(A) = 2^n$.
For large $n$, computing a set of generators for $\mu(A)$
is feasible, even when listing all elements of $\mu(A)$ is not.
\end{ex}

\begin{ex}
Suppose $A=\Z[\zeta_p]$, where $p$ is a prime and
$\zeta_p$ is a primitive $p$-th root
of unity in $\C$. Then $A$ has rank $p-1$. 
If $p>2$, then $\mu(A)=\langle \zeta_p\rangle \times \langle -1\rangle$.
\end{ex}

\section{Finite $\Q$-algebras}
\label{Qalssect}

The following two results are from commutative algebra.
These results and basic algorithms for
commutative $\Q$-algebras are given in \cite{Qalgs}.

\begin{prop}
\label{CRTforEred}
If $E$ is a  commutative $\Q$-algebra, 
then the map $$E_\sep \oplus \sqrt{0} \isom E, \quad
(x,y)\mapsto x+y$$ is an isomorphism of
$\Q$-vector spaces, and 
the natural map 
$E \to \prod_{\m\in\Spec(E)} E/\m$ induces an isomorphism
of $\Q$-algebras
$$E_{\sep} \isom \prod_{\m\in\Spec(E)} E/\m.$$
\end{prop}

In algorithms, we specify a commutative $\Q$-algebra $E$ 
by listing a system
of structure constants $a_{ijk}\in\Q$ 
that determines the multiplication in $E$ with respect to some
$\Q$-basis, just as we did for orders in the introduction.

\begin{algorithm}
\label{findmalgor}
There is a deterministic polynomial-time algorithm that 
given a commutative $\Q$-algebra $E$, computes 
a $\Q$-basis for $E_\sep \subset E$,
a $\Q$-basis for $\sqrt{0}$,
the map $E \isom E_\sep \oplus \sqrt{0}$ that is the inverse to the
first isomorphism from Proposition \ref{CRTforEred},
all $\m\in\Spec(E)$, the fields $E/\m$, and the natural maps $E \to E/\m$.
\end{algorithm}

\begin{lem}
\label{muApremark}
If $E$ is a commutative $\Q$-algebra, then:
\begin{enumerate}
\item
$
\mu(E) = \mu(E_\sep) \isom 
\bigoplus_{\m\in\Spec(E)}\mu(E/\m)$;

\item
$\mu(E)$ is finite;
\item
each $\mu(E/\m)$ is a finite cyclic group;
\item
if $\mu(E)$ has an element of order $p^k$ with  $p$ a prime,
then $\varphi(p^k) \le \dim_\Q(E)$, where $\varphi$ is
Euler's $\varphi$-function.
\end{enumerate}
\end{lem}

\begin{proof}
Part (i) holds by Proposition \ref{CRTforEred}
and the fact that 
$X^r-1$ is separable over $\Q$ for all $r\in\Z_{>0}$.
If $\mu(E)$ has an element of prime power order $p^k$, then 
$\Q(\zeta_{p^k}) \subset E/\m$ for some $\m$,
where $\zeta_{p^k}$ is a primitive $p^k$-th root of unity. 
Thus $\varphi(p^k) \le [E/\m : \Q] \le \dim_\Q(E).$
Since each $E/\m$ is a number field, 
$\mu(E/\m)$ is cyclic.

\end{proof}

\begin{algorithm}
\label{muEthmalgor}
The algorithm takes as input
a commutative $\Q$-algebra $E$ and 
produces a set of generators $S$ of $\mu(E)$ as well as
a set $R$ of defining relations for $S$.
\end{algorithm}

\begin{enumerate}
\item
For each $\nn\in\Spec(E)$, 
use the algorithm in \cite{Arjen} to
find all zeroes of $X^{r}-1$ over $E/\nn$, for 
$r = 1,2,\ldots, 2[E/\nn:\Q]^2$,
let $\zeta_{\nn}\in (E/\nn)^\ast$ be an element of
maximal order among the zeroes found, and let
$k(\nn)$ be its order.
\item
For each $\nn\in\Spec(E)$, use linear algebra to compute the unique element
$\eta_{\nn}\in E_\sep$ that under the second isomorphism from 
Proposition \ref{CRTforEred} maps to
$(1,\ldots,1,\zeta_{\nn},1,\ldots,1)\in\prod_\m \mu(E/\m)$
(with $\zeta_{\nn}$ in the $\nn$-th position).
Output 
$S = \{ \eta_\nn \in \mu(E) : \nn\in\Spec(E)\}$ 
and 
$R = \{ (0,\ldots,0,k(\nn),0,\ldots,0)\in\Z^{\Spec(E)} : \nn\in\Spec(E)\}$.
\end{enumerate}

\begin{prop}
\label{muEthmalgorworks}
Algorithm \ref{muEthmalgor}
produces correct output and runs in polynomial time.
\end{prop}

\begin{proof}
If the number field $E/\nn$ contains a primitive $r$-th root of unity, then it contains the $r$-th cyclotomic field, which has degree $\varphi(r)$ over $\Q$;
hence $\varphi(r) \le [E/\nn:\Q]$ and $r \le 2\varphi(r)^2 \le 2[E/\nn:\Q]^2.$
Together with Lemma \ref{muApremark}(i),
this implies that the algorithm is correct.
It runs in polynomial time by \cite{Arjen}.
\end{proof}

\begin{algorithm}
\label{muEthmalgor2}
The algorithm
takes as input a commutative $\Q$-algebra $E$, an element $\gamma\in E$,
and a set 
$S = \{ \eta_\nn \in \mu(E) : \nn\in\Spec(E)\}$ of generators for $\mu(E)$
as computed by Algorithm \ref{muEthmalgor}.
It tests whether $\gamma\in\mu(E)$,
and if so, finds $(a_\nn)_{\nn\in \Spec(E)}\in\Z^{\Spec(E)}$ 
with $\gamma = \prod_{\nn\in \Spec(E)}\eta_\nn^{a_\nn}$.
\end{algorithm}

\begin{enumerate}
\item
Use linear algebra to test if 
$\gamma \in E_\sep$.
If not, terminate with ``no'' (that is, $\gamma\not\in\mu(E)$).
\item
Otherwise, for each $\nn\in\Spec(E)$ 
compute the image $\gamma_\nn$ of $\gamma$ in $E/\nn$, and 
let $\zeta_\nn$ (as in Algorithm \ref{muEthmalgor}) be the
image of $\eta_\nn$ in $E/\nn$.
Try 
$a=0,1,2,\ldots,\#\mu(E/\nn) -1$ until $\gamma_\nn = \zeta_\nn^a$,
and let $a_\nn = a$.
If for some $\nn$ no $a_\nn$ exists, terminate with ``no''.
\item
Otherwise, output
$(a_\nn)_{\nn\in \Spec(E)}$.
\end{enumerate}

That Algorithm \ref{muEthmalgor2}
produces correct output and runs in polynomial time
follows from Lemma \ref{muApremark}, since 
$\mu(E/\nn) = \langle \zeta_\nn \rangle$.

\section{Orders}
\label{orderssect}

From now on, suppose that $A$ is an order.
Let 
$$
E = A_\Q = A \otimes_\Z\Q, \qquad
A_\sep = A \cap E_\sep.
$$

Since $E_\sep/A_\sep \subset E/A = A_\Q/A$ is a torsion group,
one has $E_\sep = (A_\sep)_\Q$.

\begin{lem}
\label{idmulem}
We have $\id(E_\sep)=\id(E)$,   
$\id(A_\sep) = \id(A)$, and $\mu(A_\sep) = \mu(A)$.
\end{lem}

\begin{proof}
This holds because the polynomials
$X^2-X$ and $X^r-1$ are separable over $\Q$ for all $r\in\Z_{>0}$.
\end{proof}

\begin{algorithm}
\label{Asepalgor}
The algorithm takes as input an order $A$ and it computes
the $\Q$-algebras  $E$ and $E_\sep\subset E$, 
as well as the order $A_\sep = A \cap E_\sep$,
giving 
a $\Z$-basis for $A_\sep$  
expressed both in the given
$\Z$-basis of $A$
and in the 
$\Q$-basis for $E_\sep$.
\end{algorithm}

\begin{enumerate}
\item
We use the given $\Z$-basis for $A$ as a $\Q$-basis for $E$,
with the same structure constants.
\item
Let $\pi_1 : A \to E_\sep$ and $\pi_2 : A \to \sqrt{0}$ be
the compositions of the inclusion $A\subset E$ with the map
$E \isom E_\sep \oplus \sqrt{0}$ from Algorithm \ref{findmalgor} 
followed by the natural projections to $E_\sep$ and $\sqrt{0}$,
respectively.
Using Algorithm \ref{findmalgor}, compute 
a $\Q$-basis for $E_\sep$ and 
the rational matrices
describing $\pi_1$ and $\pi_2$.
Applying the kernel algorithm in \S 14 of
\cite{HWLMSRI} to an integer multiple of the 
matrix for $\pi_2$, compute a $\Z$-basis for $A_\sep=\ker(\pi_2)$ 
expressed in the given $\Z$-basis for $A$.
Applying $\pi_1$ to this $\Z$-basis, one obtains the same
$\Z$-basis expressed in the  
$\Q$-basis for $E_\sep$.
\end{enumerate}

Algorithm \ref{Asepalgor} is clearly correct and polynomial time.

\section{Graphs attached to rings}

\begin{lem}
\label{idempotClem}
Suppose that $R$ is a commutative ring,
$\SS$ is a finite set of ideals of $R$ that are not $R$ itself,
and suppose that $\bigcap_{\a\in \SS} \a = \{ 0\}$.
Identify $R$ with its image in $\prod_{\a\in \SS} R/\a$.
Suppose that $e = (e_\a)_{\a\in \SS} \in \{ 0,1\}^\SS \subset \prod_{\a\in \SS} R/\a$.
Then $e\in R$ if and only if 
$e_\a = e_\b$ in $\{ 0,1\}$
for all $\a, \b \in \SS$ such that
$\a + \b \neq R$.
\end{lem}

\begin{proof}
First suppose $e\in R$. 
Suppose $\a, \b \in \SS$ and $\a + \b \neq R$.
Choose $e'_\a  \in \{ 0,1\} \subset R$ whose image in $R/\a$ is
$e_\a = e + \a$, and choose $e'_\b  \in \{ 0,1\} \subset R$ whose image in $R/\b$ is
$e_\b = e + \b$.
Then $e'_\a \equiv e$ mod $\a$ and $e'_\b \equiv e$ mod $\b$,
so $e'_\a \equiv e \equiv e'_\b$ mod $(\a + \b)$. 
Since $\a + \b \neq R$ we have $1\not\in \a + \b$.
Thus, $e'_\a = e'_\b$ in $\{ 0,1\}$, as desired.

Conversely, suppose that $e_\a = e_\b$ in $\{ 0,1\}$
for all $\a, \b \in \SS$ with $\a + \b \neq R$.
Let $T = \{ \a\in \SS : e_\a = 1\}$ and $U = \{ \b\in \SS : e_\b = 0\}$.
Then $\SS = T \sqcup U$.
Pick $\a\in T$ and $\b\in U$.
By our assumption, 
$\a + \b = R$. Thus, there exist $x_{\a,\b} \in\a$ and
$y_{\a,\b} \in\b$ such that $1 = x_{\a,\b} + y_{\a,\b}$.
It follows that $y_{\a,\b} \equiv 1$ mod $\a$ and
$y_{\a,\b} \equiv 0$ mod $\b$.
For all $\a\in T$, define 
$z_\a = \prod_{\b\in U}y_{\a,\b} \in R$.
Then $z_\a \equiv 1$ mod $\a$ and
$z_\a \equiv 0$ modulo each $\b\in U$.
Define $e' = 1 - \prod_{\a\in T}(1 - z_\a) \in R$.
Then $e' \equiv 1$ modulo each $\a\in T$, and
$e' \equiv 0$ modulo each $\b\in U$.
Thus, $e' \equiv e_\a$ mod $\a$ for each $\a\in \SS$, so $e'=e$.
\end{proof}

We say that $D$ is an order in a separable $\Q$-algebra if
$D$ is an order and $D_\Q=D\otimes_\Z\Q$ is separable.

\begin{defn}
\label{graphdef}
Suppose that $D$ is an order in a separable $\Q$-algebra $D_\Q$.
For $\m, \nn \in\Spec(D_\Q)$ with $\m \neq \nn$,
let 
$$
n(D,\m,\nn)=\#(D/((\mm\cap D) + (\nn\cap D))), 
$$ 
and let $\Gamma(D)$ denote the graph on $\Spec(D_\Q)$ defined by connecting distinct
vertices $\m, \nn \in\Spec(D_\Q)$   
by an edge
if and only if $n(D,\m,\nn) > 1$.
\end{defn}

\begin{lem}
\label{nDlem}
$n(D,\m,\nn) \in \Z_{>0}$.
\end{lem}

\begin{proof}
Let
$
R = D/((\mm\cap D) + (\nn\cap D)).
$
Then $n(D,\m,\nn) =\# R$.
Letting $-_\Q = -\otimes_\Z \Q$,
we have 
$$
R_\Q = D_\Q/((\mm_\Q\cap D_\Q) + (\nn_\Q\cap D_\Q))
= D_\Q/(\m +\nn) =0
$$
so $R$ is torsion. Since $R$ is finitely generated
as an abelian group, it is finite,
so $n(D,\m,\nn) \in \Z_{>0}$.
\end{proof}

\begin{ex}
Let $r\in\Z[X]$ be monic. Then $D=\Z[X]/(f)$ is an order in a
separable $\Q$-algebra if and only if $f$ is squarefree.
Suppose $f$ is squarefree. Then $D_\Q =\Q[X]/(f)$, and
$\Spec(D_\Q)$ is in bijection with the set of monic irreducible
factors $g$ of $f$ in $\Z[X]$, each $g$ corresponding to
$\m = (g)/(f)$. If $g,h$ correspond to $\m,\nn$, respectively,
then $n(D,\m,\nn) = |R(g,h)|$,
with $R$ denoting the resultant.
\end{ex}

Suppose $D$ is an order in a separable $\Q$-algebra.
It is natural to ask whether the decomposition 
$D_\Q \isom \prod_{\m\in\Spec(D_\Q)}D_\Q/\m$ (Proposition \ref{CRTforEred})
gives rise to
a decomposition of the order $D$. This depends on the idempotents
that are present in $D$. The graph $\Gamma(D)$ tells us which idempotents
occur in $D$ (see Lemma \ref{idempotClem} and Proposition \ref{bijgraphprop}).

\begin{notation}
\label{eWdefn}
Suppose that $D$ is an order in a separable $\Q$-algebra.
If $W \subset \Spec(D_\Q)$, define 
$$
e_W = (e_\m)_{\m\in \Spec(D_\Q)} \in\id(\prod_{\m\in \Spec(D_\Q)} D_\Q/\m) = \{0,1\}^{\Spec(D_\Q)}
$$
by $e_\m = 1$ if $\m\in W$ and $e_\m = 0$ if $\m\not\in W$.
\end{notation}

\begin{algorithm}
\label{graphalgor}
The algorithm takes an order $D$ in a separable $\Q$-algebra 
and computes the graph $\Gamma(D)$, 
 its connected components,
and its weights $n(D,\m,\nn)$ for all $\m,\nn\in\Spec(D_\Q)$.
\begin{enumerate}
\item
Use Algorithm \ref{findmalgor} to compute $\Spec(D_\Q)$ and
the maps $D_\Q \to D_\Q/\m$ for $\m\in\Spec(D_\Q)$.
\item
For each $\m\in\Spec(D_\Q)$ compute $\m\cap D = \ker(D \to D_\Q/\m)$
by applying the kernel algorithm in \S 14 of
\cite{HWLMSRI}.
\item
For all $\m\neq \nn\in\Spec(D_\Q)$,
apply the image algorithm in \S 14 of
\cite{HWLMSRI} to compute a $\Z$-basis of 
$$
\image((\m \cap D) \oplus (\nn \cap D) \to D) = 
(\m \cap D) + (\nn \cap D)
$$
expressed in a $\Z$-basis of $D$, and compute
$n(D,\m,\nn)$ as the absolute value of the determinant of the
matrix whose columns are those basis vectors.
\item
Use the numbers $n(D,\m,\nn)$ to obtain the graph $\Gamma(D)$
and its connected components.
\end{enumerate}
\end{algorithm}

The algorithm runs in polynomial time by well-known graph algorithms
(see for example \cite{HopcroftTarjan}).

\begin{prop}
\label{bijgraphprop}
Suppose that $D$ is an order in a separable $\Q$-algebra. 
\begin{enumerate}
\item
Suppose $e=(e_\m)_{\m\in\Spec(D_\Q)} \in\id(\prod_\m D_\Q/\m) = \{0,1\}^{\Spec(D_\Q)}$.
Then the following are equivalent:
\begin{enumerate}
\item
$e\in D$,
\item
$e_\m = e_\nn$ whenever $\m$ and $\nn$ are connected
in $\Gamma(D)$,
\item
$e_\m = e_\nn$ whenever $\m$ and $\nn$ are in the same connected
component of $\Gamma(D)$.
\end{enumerate}
\item 
Let $\Omega$ denote the set of connected components of the graph $\Gamma(D)$
and recall $e_W$ from Definition \ref{eWdefn}.
Then $W \mapsto e_W$ gives a bijection
$$
\Omega \isom \prid(D) \subset D \subset \prod_{\m\in \Spec(D_\Q)}D_\Q/\m.
$$ 
\end{enumerate}
\end{prop}

\begin{proof}
Apply Lemma \ref{idempotClem} with $R = D$
and $S=\{ \m\cap D : \m\in\Spec(D_\Q)\}$.
We have $\bigcap_{\a\in \SS} \a =\bigcap_{\m}(\m\cap D) = \{ 0\}$
since $D$ injects into $\prod_{\m}D_\Q/\m$.
Identifying $\id(\prod D_\Q/\m)$ with $\{ 0,1\}^\SS$,
Lemma \ref{idempotClem} implies that if
$e = (e_\m)_{\m\in \Spec(D_\Q)} \in \id(\prod D_\Q/\m)$,
then $e\in D$ if and only if
$e_\m = e_\nn$ for all $\m,\nn\in\Spec(D_\Q)$ 
that are connected in $\Gamma(D)$.
It follows that for each $e = (e_\m)_{\m} \in \id(D)$
the components $e_\m$ are constant ($0$ or $1$) on each connected
component of $\Gamma(D)$.
Part (i) now follows.
It also follows that there is a bijection
$$
\{ \text{subsets of $\Omega$}\} \to \id(D)
$$
defined by
$T \mapsto \sum_{W\in T}e_W$ with inverse 
$e=(e_\m)_\m \mapsto \{ W\in\Omega : e_\m=1 \text{ for all $\m\in W\}$}$.
Under this bijection, $\prid(D)$ corresponds to $\Omega$,
and this gives the bijection in (ii).
\end{proof}

\begin{rem}
In particular, by Proposition \ref{bijgraphprop}(ii) an order $D$ is connected if and only if $\Gamma(D)$ is connected.
\end{rem}

\section{Finding idempotents}
\label{primidsect}

The set of idempotents of an order may be too large to compute,
but the set of primitive idempotents is something that we are able
to efficiently compute.

\begin{algorithm}
\label{idempotentalg}
Given an order $A$, the algorithm outputs the set of primitive idempotents of $A$.
\end{algorithm}

\begin{enumerate}
\item
Use Algorithm \ref{Asepalgor}
to compute $A_\sep$. 
\item
Use Algorithm \ref{graphalgor}
to compute the graph $\Gamma(A_\sep)$ and its connected components.
\item
For each connected component $W$ of 
$\Gamma(A_\sep)$, with $e_W \in \{ 0,1\}^{\Spec(E)}
\subset \prod_{\m\in\Spec(E)}E/\m$ as in Notation \ref{eWdefn},
use the inverse of the square matrix with $\Q$-coefficients 
that gives the natural map
$E_\sep \isom \prod_{\m\in\Spec(E)}E/\m$ 
of Proposition \ref{CRTforEred}
to lift $e_W$ to $E_\sep$. Output these lifts.
\end{enumerate}

If follows from Proposition \ref{bijgraphprop}(ii)
that the lift $e_W$ to $E_\sep$ is in $A_\sep$,
and that Algorithm \ref{idempotentalg} gives the desired output $\prid(A)$.
It is clear that it runs in polynomial time. 

\section{Discrete logarithms}
\label{DLsect}

In this section, we suppose that $G$ is a multiplicatively written abelian group 
with elements represented by finite bitstrings.
All algorithms in the present section have $G$ as part of their
input. Thus, saying that they are polynomial-time means that their runtime 
is bounded by a polynomial function of the length of the parameters
specifying $G$ plus the length of the rest of the input.
We suppose that polynomial-time 
algorithms for the group operations and for equality testing
in $G$ are available. 

\begin{defn}
\label{effpresdefn}
We say  $\langle S|R\rangle$ is an {\bf efficient presentation}
for $G$ if $S$ is a finite set, and we have a map $f = f_S : S\to G$ satisfying:
\begin{enumerate}
\item[(a)]
$f(S)$ generates $G$, i.e., the map $g_S : \Z^S \to G$,
$(b_s)_{s\in S} \mapsto \prod_{s\in S}f(s)^{b_{s}}$
is surjective,
\item[(b)]
$R\subset\Z^S$ is a finite set of generators for $\ker(g_S)$,
\item[(c)]
we have a polynomial-time algorithm 
that on input $\gamma\in G$ finds an element of $g_S^{-1}(\gamma)$
(i.e., finds $(c_s)_{s\in S} \in\Z^S$ such that $\gamma = \prod_{s\in S}f(s)^{c_{s}}$).
\end{enumerate}
\end{defn}

\begin{notation}
Suppose $\langle S|R\rangle$ is an efficient presentation for $G$.
Define 
$$
\rho : \Z^R \to \Z^S, \quad
\rho((m_r)_{r\in R}) = \sum_{r\in R}m_r r.
$$
Suppose $T$ is a finite set and we have a map $f_T : T \to G$.
By abuse of notation we usually
suppress the maps $f_S$ and $f_T$ and write $s$ for $f_S(s)$
and $f_T(s)$ and write $\langle T \rangle$ for $\langle f_T(T) \rangle$.
Define
$$g_T : \Z^T \to \langle T \rangle, 
\quad (b_t)_{t\in T} \mapsto \prod_{t\in T}t^{b_t}.$$
Define
$h = h_{T}
: \Z^T \to \Z^S$ by using (c)
to write each $t\in T$ as $t=\prod_{s\in S}s^{c_{s,t}}$
and defining 
$$
h((b_t)_{t\in T}) = (\sum_{t\in T}b_tc_{s,t})_{s\in S}\in\Z^S
$$
so that $g_T=g_S\circ h$.
\end{notation}

For the remainder of this section
we suppose that an efficient presentation $\langle S|R\rangle$ for an 
abelian group $G$ is given.

\begin{algorithm}
\label{genlprinalgor}
The algorithm takes as input $G$, 
an efficient presentation $\langle S|R\rangle$ for $G$, and
a finite set $T$ with a map $T \to G$,
and outputs a finite set $U = U_T$ of generators for $\ker(g_T)$.
\end{algorithm}

\begin{enumerate}
\item
Define
$h-\rho : \Z^T \times \Z^R \to \Z^S$ by $(h-\rho)(x,y)= h(x)-\rho(y)$. 
Use the kernel algorithm in \S 14 of
\cite{HWLMSRI} to compute a finite set $V$ of generators for $\ker(h-\rho)$.
\item
Compute the image $U$ of $V$ under the projection map
$\Z^T \times \Z^R \onto \Z^T$, $(x,y)\mapsto x$.
\end{enumerate}

\begin{thm}
\label{genlprinthm}
Algorithm \ref{genlprinalgor}
produces correct output and runs in polynomial time.
\end{thm}

\begin{proof}
We have:
\begin{align*}
x\in\ker(g_T) & \iff h(x)\in\ker(g_S)=\im(\rho) \\
& \iff 
\exists y\in\Z^R \text{ such that }  h(x)=\rho(y)  \\
& \iff 
\exists y\in\Z^R \text{ such that }  (h-\rho)(x,y)=0  \\
& \iff
\exists y\in\Z^R \text{ such that }  (x,y) \in \langle V \rangle  \\
& \iff
x \in \proj(\langle V \rangle) = \langle \proj(V) \rangle = \langle U \rangle.
\end{align*}
\end{proof}

\begin{algorithm}
\label{genlprinalgor2}
The algorithm takes as input  $G$, 
an efficient presentation $\langle S|R\rangle$ for $G$,
a finite set $T$ with a map $T \to G$, and an element $\gamma\in G$,
and decides whether $\gamma\in \langle T \rangle$, and if it is,
produces an element of $g_T^{-1}(\gamma)$ 
(i.e., finds $(c_t)_{t\in T} \in\Z^T$ such that $\gamma = \prod_{t\in T}t^{c_{t}}$).
\end{algorithm}

\begin{enumerate}
\item
Apply 
Algorithm \ref{genlprinalgor} with $T \cup \{\gamma\}$  
in place of $T$ to find
a finite set of generators 
$U_{T \cup \{\gamma\}} \subset \Z^{T \cup \{\gamma\}}$
for 
$\ker(g_{T \cup \{\gamma\}})$, where
$$g_{T \cup \{\gamma\}} : 
\Z^{T \cup \{\gamma\}} = \Z^T\times\Z^{\{\gamma\}}\to G, \qquad
(x,n)\mapsto g_T(x)\gamma^n.$$
\item
Map the elements $u\in U_{T \cup \{\gamma\}} \subset \Z^{T \cup \{\gamma\}}
= \Z^{T} \times \Z^{\{\gamma\}}$ 
to their 
$\Z^{\{\gamma\}}$-components $u(\gamma)\in\Z$.
If $\sum_{u\in U_{T \cup \{\gamma\}}}u(\gamma)\Z \neq \Z$ then $\gamma\not\in \langle T \rangle$;
if $1=\sum_{u\in U_{T \cup \{\gamma\}}}n_u u(\gamma)$ with $(n_u)_{u\in U_{T \cup \{\gamma\}}}\in\Z^{U_{T \cup \{\gamma\}}}$
then $\gamma\in \langle T \rangle$ and
the $\Z^T$-component of 
$-\sum_{u\in U_{T \cup \{\gamma\}}}n_u u \in \Z^{T \cup \{\gamma\}}= \Z^{T} \times \Z^{\{\gamma\}}$ is in
$g_T^{-1}(\gamma)$.
\end{enumerate}

\begin{algorithm}
\label{genlprinalgor2a}
The algorithm takes as input  $G$, 
an efficient presentation $\langle S|R\rangle$ for $G$, and
a finite set $T$ with a map $T \to G$, and 
outputs an efficient presentation 
$\langle T|U_{T} \rangle$ for $\langle T \rangle$.
\end{algorithm}

\begin{enumerate}
\item
Apply Algorithm \ref{genlprinalgor} to obtain a set $U_T$ of relations.
\item
Output the presentation $\langle T|U_{T} \rangle$.
\end{enumerate}

\begin{thm}
\label{genlprinthm2}
Algorithms \ref{genlprinalgor2} and \ref{genlprinalgor2a}
produce correct output and run in polynomial time.
In particular, if one has an efficient presentation for $G$, and
$T$ is a finite set with a map $T \to G$, then
$\langle T|U_{T} \rangle$ is an efficient presentation
for $\langle T \rangle$.
\end{thm}

\begin{proof}
We have:
\begin{align*}
\gamma \in \langle T \rangle & \iff 
\exists x\in\Z^T \text{ such that } \gamma = g_T(x)  \\
& \iff \exists x\in\Z^T  \text{ such that }  (-x,1) \in\ker(g_{T \cup \{\gamma\}}:\Z^T\times\Z\to G) =  
\langle U_{T \cup \{\gamma\}} \rangle \\
& \iff 
1 \in\im(\proj : \langle U_{T \cup \{\gamma\}} \rangle \subset \Z^T\times\Z\to \Z) \\
& \iff 
\exists (n_u)_{u\in U_{T \cup \{\gamma\}}}, \exists x\in\Z^T  \text{ such that }  \sum_u n_u u = (-x,1) 
\end{align*}
where $\proj$ is projection onto the second component.
\end{proof}

\begin{algorithm}
\label{genlprinalgor4}
The algorithm takes as input  $G$, 
an efficient presentation $\langle S|R\rangle$ for $G$,
finite sets $T$ and $T'$, and maps $f_T: T \to G$ and $f_{T'}: T'\to G$, 
and outputs  a finite set of generators for the kernel of the composition
$\Z^T \to G \to G/\langle T' \rangle$, where $\Z^T \to G$ is the map $g_T$.
\end{algorithm}

\begin{enumerate}
\item
Apply Algorithm \ref{genlprinalgor} to the finite set
$T \sqcup T'$ and the map $T \sqcup T' \to G$ obtained from 
$f_T$ and $f_{T'}$, to obtain
generators for the kernel of the map
$$\Z^T \times \Z^{T'} = \Z^{T \sqcup T'} \to G, 
\qquad (x,y) \mapsto g_T(x) - g_{T'}(y).$$
\item
Project these generators 
to their $\Z^T$-component.
\end{enumerate}

\begin{thm}
\label{genlprinthm4}
Algorithm \ref{genlprinalgor4}
produces correct output and runs in polynomial time.
\end{thm}

\begin{proof}
We have:
\begin{align*}
x \in \ker(\Z^T \to G/\langle T' \rangle)
& \iff 
g_T(x) \in \langle T' \rangle = \im(g_{T'}) \\
& \iff 
\exists y\in\Z^{T'}  \text{ such that }   g_T(x) = g_{T'}(y)  \\
& \iff 
\exists y\in\Z^{T'}  \text{ such that }   (x,y) \in\ker(\Z^T\times\Z^{T'} \to G)  \\
& \iff 
x \in \proj(\ker(\Z^T\times\Z^{T'} \to G) \to \Z^T)
\end{align*}
where 
$\proj$ denotes projection onto the $\Z^T$-component.
\end{proof}

\noindent{\bf{Proof of Theorem \ref{muAthm2}.}}
One starts by computing $E=A\otimes_\Z\Q$, using the same
structure constants as for $A$.
Algorithm \ref{muEthmalgor} produces a presentation for $\mu(E)$,
and by Algorithm \ref{muEthmalgor2} this is an efficient presentation.
Given $T$ and $\zeta$ as in Theorem \ref{muAthm2},
one can test whether $\zeta\in E$ by Algorithm \ref{muEthmalgor2}.
Now Theorem \ref{muAthm2} is obtained from Algorithm \ref{genlprinalgor2},
with $G=\mu(E)$ and $\gamma=\zeta$.

\section{Separable polynomials over connected rings}
\label{sepconnsect}

Proposition \ref{degbd}(b) will be used to prove Proposition \ref{Wconnprop} below.

\begin{prop}
\label{degbd}
Suppose $R$ is a connected commutative ring, $f \in R[X]$, and
$R[X] f + R[X] f' = R[X]$. Then:
\begin{enumerate}
\item[\rm (a)]
if $r,s\in R$ and $f(r) = f(s) = 0$, then
$r-s \in \{ 0 \} \cup R^\ast$;
\item[\rm (b)]
if $S$ is a non-zero ring and $\varphi : R \to S$ is a ring homomorphism,
then the restriction of $\varphi$ to $\{ r\in R : f(r) = 0 \}$ is injective;
\item[\rm (c)]
$f\neq 0$ and $\#\{ r\in R : f(r) = 0 \} \le \deg(f)$.
\end{enumerate}
\end{prop}

\begin{proof}
Suppose $f(r) = f(s) = 0$. Write $f=(X-r)g$ 
and $1 = hf+kf'$ 
with $g,h,k\in R[X]$.
Then $g(r) = f'(r)\in R^\ast$. 
Since $g(s)\equiv g(r)$ mod $(r-s)R$ we can write
$g(s) = g(r) + (r-s)t$ with $t\in R$.
Thus, $0 = f(s) = (s-r)g(s)= (s-r)(g(r) + (r-s)t)$,
so 
\begin{equation}
\label{sreqn}
(s-r)g(r)= t(s-r)^2.
\end{equation}
Thus, $t\cdot (s-r) \cdot g(r)^{-1} = (t\cdot (s-r)\cdot g(r)^{-1})^2$, an idempotent.
If $t\cdot (s-r) \cdot g(r)^{-1} = 0$, then by \eqref{sreqn} we have 
$(s-r)g(r)= 0$, and thus  $r-s=0$ since $g(r) \in R^\ast$.
If $t\cdot (s-r) \cdot g(r)^{-1} = 1$, then $r-s \in R^\ast$.
This gives (a).

For (b), suppose $r,s\in R$, $r \neq s$, and $f(r) = f(s) = 0$.
By (a) we have 
$r-s \in R^\ast$. Since $\varphi(1) = 1 \neq 0$, we have $\varphi(r-s) \neq 0$.

For (c), let $\m$ be a maximal ideal of $R$.
Then $R \to R/\m$ induces a map
$$
\{ r\in R : f(r)=0\} \to \{ u\in R/\m : (f \text{ mod $\m)(u)=0\}$}
$$
that is injective by (b). 
Since $R/\m$ is a field and $f \text{ mod $\m \in (R/\m)[X]$}$ is non-zero, 
we have 
$$
\#\{ r\in R : f(r) = 0 \} \le \deg(f \text{ mod $\m)  \le \deg(f)$}.
$$
\end{proof}

\begin{cor}
\label{cycliclemma}
Suppose $R$ is a connected commutative ring, 
$m\in \Z_{>0}$, and
$m\cdot 1\in R^\ast$.
Then $\{ \zeta\in R : \zeta^m=1\}$ is a cyclic subgroup of $R^\ast$ whose
order divides $m$.
\end{cor}

\begin{proof}
Applying Proposition \ref{degbd} with $f = X^m-1$ 
gives that the subgroup has
order dividing $m$. 
Applying Proposition \ref{degbd} with $f = X^d-1$ for each
divisor $d$ of $m$ gives that this abelian subgroup has at most $d$ elements
of order dividing $d$, and thus is cyclic. 
\end{proof}

\section{From $\mu(E)$ to $\mu(B)$}
\label{EtoBsect}

Fix an order $A$. Recall that 
$E = A_\Q = A \otimes_\Z\Q$ 
and  
$A_\sep = A \cap E_\sep$.
For $\m\in\Spec(E)$, the image of $A_\sep$ in $E/\m$ may be identified
with $A_\sep/(\m\cap A_\sep)$; it is a ring of which the
additive group is a finitely generated subgroup of the 
$\Q$-vector space $E/\m$, so it is an order.
We now write
\begin{equation}
\label{Bdefn} 
B = \prod_{\m\in\Spec(E)} A_\sep/(\m\cap A_\sep).
\end{equation}
This is an order in $\prod_{\m\in\Spec(E)} E/\m$.
We identify $A_\sep$ with its image in $B$ under the map
$$
E_\sep \isom \prod_{\m\in\Spec(E)} E/\m
$$
and identify $B$ with a subring of 
$E_\sep$ using the same map.
One has
$$
A_\sep  \subset B \subset E_\sep.
$$
Since the abelian group $B/A_\sep$ is both torsion and 
finitely generated, it is finite, and one has $B_\Q = E_\sep$.
The graph $\Gamma(B)$ consists of the vertices $\m\in\Spec(E)$
and no edges.

\begin{prop}
\label{Balgorrrr}
There is a deterministic polynomial-time algorithm that, given an
order $A$, computes a $\Z$-basis for $A_\sep/(\m\cap A_\sep)$
in $E/\m$ for every $\m\in\Spec(E)$, a $\Z$-basis for 
$B$ in $E_\sep$, and the index $(B:A_\sep)$.
\end{prop}

\begin{proof}
One simply computes a $\Z$-basis for $A_\sep$ as in 
Algorithm \ref{Asepalgor}, and a $\Z$-basis for 
the image of the map $A_\sep \subset E_\sep \to E/\m$
using the image algorithm in \S 14 of \cite{HWLMSRI},
for each $\m\in\Spec(E)$.
Combining these bases for all $\m$ and applying the
inverse of the second isomorphism in Proposition \ref{CRTforEred}
one finds a $\Z$-basis for $B$ in $E_\sep$.
The index $(B:A_\sep)$ is the absolute value of the determinant
of any matrix expressing a $\Z$-basis for $A_\sep$ in a $\Z$-basis for $B$.
\end{proof}

\begin{prop}
\label{Balgorrrr2}
For each order $A$ and each $\m\in\Spec(E)$ the group
$\mu(A_\sep/(\m\cap A_\sep))$ is finite cyclic.
Also, there is a deterministic polynomial-time algorithm that, given 
$A$ and $\m$, computes a 
generator $\theta_{\m}$ of $\mu(A_\sep/(\m\cap A_\sep))$,
its order, the complete prime factorization of its order,
and, for each prime number $p$ 
a generator $\theta_{\m,p}$ for $\mu(A_\sep/(\m\cap A_\sep))_p$.
\end{prop}

\begin{proof}
The first statement follows from Lemma \ref{muApremark}(iii).
For $\theta_{\m}$ one can take the first power of the generator $\zeta_\m$ of
$\mu(E/\m)$ found in Algorithm \ref{muEthmalgor} that belongs
to $A_\sep/(\m\cap A_\sep)$, i.e., for which all coordinates on a $\Z$-basis
of $A_\sep/(\m\cap A_\sep)$ (which is a $\Q$-basis of $E/\m$)
are integers. The order of $\theta_{\m}$ is then easy to write down, and since
the prime numbers dividing that order are, by Lemma \ref{muApremark}(iv),
bounded by $1+\rk_\Z(A)$, it is also easy to factor into primes.
If $p^k$ is a prime power exactly dividing $\order(\theta_{\m})$,
one can take $\theta_{\m,p}=\theta_{\m}^{\order(\theta_{\m})/p^k}$.
\end{proof}

\begin{prop}
\label{Balgorrrr3}
There is a deterministic polynomial-time algorithm that, given an
order $A$, determines all prime factors $p$ of $\#\mu(B)$,
with $B$ as in \eqref{Bdefn}, as well as
an efficient presentation for $\mu(B)$ 
and, for each $p$, an efficient presentation for $\mu(B)_p$.
\end{prop}

\begin{proof}
This follows directly from Proposition \ref{Balgorrrr2}
and the isomorphisms 
$$\mu(B) \cong \prod_{\m\in\Spec(E)} \mu(A_\sep/(\m\cap A_\sep)) \quad
\text{ and } \quad
\mu(B)_p \cong \prod_{\m\in\Spec(E)} \mu(A_\sep/(\m\cap A_\sep))_p$$
in the same way as for $\mu(E)$ in section \ref{Qalssect}.
\end{proof}

\section{From $\mu(B)_p$ to $\mu(C)_p$}
\label{BtoCsect}

Let $A$, $E$, $A_\sep$, and $B$ be as in the previous section,
and fix a prime number $p$.
Let
\begin{equation}
\label{Cdefn}
C = A_\sep\left[1/p\right] \cap B.
\end{equation}

We have
$$
A_\sep \subset C \subset B \subset E_\sep
$$
so $C$ is an order with $C_\Q=E_\sep$, and
$$
C = \{ x\in B : p^ix\in A_\sep \text{ for some $i\in\Z_{\ge 0}\}$}.
$$
The group $C/A_\sep$
is finite of $p$-power order, and the group $B/C$ is finite of order
prime to $p$. 
These orders can be quickly computed from the order of $B/A_\sep$
computed in Proposition \ref{Balgorrrr}.
We emphasize that $C$ depends on $p$.

Let $t = (B:C)$. Then $C/A_\sep = t(B/A_\sep)$, so
$C=tB+A_\sep$, which is the image of the map
$B \oplus A_\sep \to B$, $(x,y) \mapsto tx+y$.
Thus one can find a $\Z$-basis for $C$ from the image algorithm
in \S 14 of \cite{HWLMSRI}.

\begin{prop}
\label{nonpcptprop}
Suppose that $A$ is an order and $p$ is a prime.
Suppose $\m,\nn\in\Spec(E)$ with $\m\neq\nn$.
Then:
\begin{enumerate}
\item
$C/((\m\cap C) + (\nn\cap C))$ is the non-$p$-component of
$A_\sep/((\m\cap A_\sep) + (\nn\cap A_\sep))$;
\item 
$\m$ and $\nn$ are connected in $\Gamma(C)$ if and only if
$n(A_\sep,\m,\nn) \not\in p^{\Z_{\ge 0}}$.
\end{enumerate}
\end{prop}

\begin{proof}
For $Z = A_\sep, B$, and $C$, write
$\tilde{Z}$ for the finite abelian group
$Z/((\m\cap Z) + (\nn\cap Z))$ (cf.~Lemma \ref{nDlem}).
Let $p^r = (C:A_\sep)$ and $t = (B:C)$.
Then $\gcd(p^r,t)=1$.
Since 
$\Gamma(B)$ has no edges, we have
$(\m\cap B) + (\nn\cap B) = B$, so $\tilde{B}=0$.
Consider the maps 
$
\xymatrix{\tilde{A}_\sep \ar@<.5ex>^-{1}[r] & \tilde{C}\ar@<.5ex>^-{p^r}[l] 
\ar@<.5ex>^-{1}[r] & \tilde{B}=0\ar@<.5ex>^-{t}[l]  }
$
where a map $\xymatrix{\tilde{Z}_1 \ar^-{d}[r] & \tilde{Z}_2}$
is the map induced by multiplication by $d$ on $Z_1$.
(The maps are well-defined since $A_\sep \subset C \subset B$ and
$p^r C \subset A_\sep$ and $tC \subset B$.)

Since $\tilde{B}=0$, taking the composition
$
\xymatrix{\tilde{C}
\ar^-{1}[r] & \tilde{B} \ar^-{t}[r] & \tilde{C} }
$
shows that $t\tilde{C} = 0$.
If $x\in \tilde{C}$ and $p^rx=0$, then since $\gcd(p^r,t)=1$ we have
$x=0$. Thus, the composition
$
\xymatrix{\tilde{C}
\ar^-{p^r}[r] & \tilde{A}_\sep \ar^-{1}[r] & \tilde{C} }
$
is an injection, and thus an automorphism $\alpha$ of the finite abelian group $\tilde{C}$.
It follows that 
$
\xymatrix{\tilde{A}_\sep  \ar^-{1}[r] & \tilde{C} }
$
is surjective and
$
\xymatrix{\tilde{C} \ar^-{p^r}[r] & \tilde{A}_\sep}
$
is injective.
Further, letting $\tilde{A}_\sep[p^r]$ denote the kernel of
multiplication by $p^r$ in $\tilde{A}_\sep$, we have
$$\ker(\xymatrix{\tilde{A}_\sep  \ar^-{1}[r] & \tilde{C} }) = 
\ker(\xymatrix{\tilde{A}_\sep  \ar^-{1}[r] & \tilde{C} 
\ar^-{p^r}[r] & \tilde{A}_\sep}) = \tilde{A}_\sep\left[p^r\right].$$
This gives a split short exact sequence
$$
\xymatrix{0 \ar[r] & \tilde{A}_\sep\left[p^r\right] \ar[r] & 
\tilde{A}_\sep \ar@<.5ex>^-{1}[r]  & 
\tilde{C}\ar@<.5ex>^-{p^r\alpha^{-1}}[l] \ar[r] & 0 }
$$
with $\tilde{C}$ killed by $t$.
Thus $\tilde{C}$ is the non-$p$-component of
$\tilde{A}_\sep$, proving (i).

We have $n(A_\sep,\m,\nn) \not\in p^{\Z_{\ge 0}}$
if and only if 
$\tilde{A}_\sep$ is not a $p$-group, i.e.,
if and only if 
$\tilde{C} \neq 0$ (by (i)).
But  $\tilde{C} \neq 0$ 
if and only if 
$\m$ and $\nn$ are connected in $\Gamma(C)$.
This gives (ii).
\end{proof}

One could compute $\Gamma(C)$ by applying 
Algorithm \ref{graphalgor} with $D=C$.
Thanks to Proposition \ref{nonpcptprop} 
we can compute $\Gamma(C)$ without
actually computing $C$, as follows.

\begin{algorithm}
\label{graphCalgor}
The algorithm takes an order $A$ and the numbers $n(A_\sep,\m,\nn)$,
and computes the graph $\Gamma(C)$
and its connected components.
\begin{enumerate}
\item
Connect two vertices $\m$ and $\nn$  
if and only if
$n(A_\sep,\m,\nn) \not\in p^{\Z_{\ge 0}}$.
\item
Output the associated graph and the connected components.
\end{enumerate}
\end{algorithm}

\begin{defn}
\label{CsubWdef}
If $W \subset \Spec(E)$, let $C_W$ denote the image of $C$
in the quotient $$\prod_{\m\in W}A_\sep/(\m\cap A_\sep)$$ of $B$.
\end{defn}

\begin{lem}
\label{Cprodlem}
Let $\Omega$ denote the set of connected components of the graph $\Gamma(C)$.
Then the natural map 
$F : C \to \prod_{W\in\Omega} C_W$ is an isomorphism. 
\end{lem}

\begin{proof}
The map 
$F$ is injective,
since 
$$
C \subset B = \prod_{W\in\Omega}\prod_{\m\in W} A_\sep/(\m\cap A_\sep).
$$
If $f_W : C \onto C_W$ is the natural map, $e_W$ is as defined in
Notation \ref{eWdefn} with $D=C$, and
$x = (f_W(c_W))_{W\in\Omega}$ is an arbitrary element
of $\prod_{W\in\Omega} C_W$,
then $F(\sum_{W\in\Omega} c_W e_W) =x$, so $F$ is surjective.
The result now follows from Proposition \ref{bijgraphprop}(ii).
\end{proof}

\begin{prop}
\label{Wconnprop}
Suppose $A$ is an order and $p$ is a prime number.
Recall $C$ as defined in \eqref{Cdefn}.
Fix a subset $W \subset \Spec(E)$ 
for which the induced subgraph of $\Gamma(C)$ is connected.
Then:
\begin{enumerate}
\item 
the ring $C_W$ is connected,
\item
the natural map $\mu(C_W)_p \to \mu(C_{\{ \m\}})_p$ is injective for all $\m\in W$,
\item
the group $\mu(C_W)_p$ is cyclic,
\item
if $W'$ is a non-empty subset of $W$, then the natural map
$\mu(C_W)_p \to \mu(C_{W'})_p$ is injective.
\end{enumerate}
\end{prop}

\begin{proof}
Part (i) follows from Lemma \ref{idempotClem}.

Let $B_W = \prod_{\m\in W}A_\sep/(\m\cap A_\sep).$
We have 
$$\id(C_W\left[{1/p}\right]) \subset   \id\left(\prod_{\m\in W}E/\m\right)
= \id(B_W).$$
Recall $B$ from \eqref{Bdefn}.
Since $(B:C)$ is coprime to $p$, so is $(B_W:C_W)$.
Suppose $e\in \id(C_W\left[{1/p}\right])$. Then $e\in \id(B_W)$ and
there exists $m\in\Z - p\Z$ such that $me\in C_W$
(e.g., $m=(B_W:C_W)$).
Further, there exists $k\in \Z_{\ge 0}$ such that $p^ke\in C_W$.
Since $m$ and $p^k$ are coprime, we have $e\in C_W$.
Thus, $\id(C_W\left[{1/p}\right]) = \id(C_W) =\{ 0,1\}$,
so $C_W\left[{1/p}\right]$ is connected.
Now by Corollary \ref{cycliclemma}  with 
$R = C_W\left[{1/p}\right]$ and 
$m=\#\mu(C_W\left[{1/p}\right])_p$,
the group $\mu(C_W\left[{1/p}\right])_p$ is cyclic, so its subgroup
$\mu(C_W)_p$ is cyclic as well, which is (iii).
Also, by Proposition \ref{degbd}(b) with
$R = C_W\left[{1/p}\right]$ and $f = X^{m} - 1$, 
the map
$\mu(C_W\left[{1/p}\right])_p \to \mu(C_{W'}\left[{1/p}\right])_p$
is injective for each non-empty $W'\subset W$.
This implies (iv). With $W' = \{ \m\}$ one obtains (ii).
\end{proof}

\begin{rem}
If $A$ is a connected  
order in a separable $\Q$-algebra
and $p$ is a prime number that does not divide $\#(B/A)$,
then $\mu(A)_p$ is cyclic. 
This follows from Proposition \ref{Wconnprop}(iii);
$C=A$ since $E=E_\sep$
and $p\nmid\#(B/A)$, and one can take $C=C_W$
since $A$ is connected.
\end{rem}

By Proposition \ref{Wconnprop}(ii,iii),
if $W$ is a connected component of $\Gamma(C)$, then 
the natural map $$\mu(C_W)_p \to \mu(A/(\m\cap A))_p$$ is
injective for all $\m\in W$, and 
$\mu(C_W)_p$ is cyclic.
This gives an efficient algorithm for computing
$\mu(C_W)_p$, and thus a set of generators for
$\mu(C)_p$, as follows.

\begin{algorithm}
\label{BtoCalgor}
Given an order $A$ and a prime $p$, 
the algorithm finds an efficient presentation for $\mu(C)_p$.
\begin{enumerate}
\item
Apply Algorithm \ref{Balgorrrr2} to compute 
a generator of the cyclic group 
$\mu(A_\sep/(\m \cap A_\sep))_p$ for each $\m\in\Spec(E)$.
\item
Apply Algorithm \ref{graphCalgor} to compute $\Gamma(C)$ and
its connected components $W$.
\item
For each $W$, do the following:
\begin{enumerate}
\item
Apply the image algorithm in \S 14 of
\cite{HWLMSRI} to compute a basis for the order 
$$C_W = \image(C \to\prod_{\m\in W}E/\m).$$
\item
Pick $\m_1\in W$ with $\#\mu(A_\sep/(\m_1 \cap A_\sep))_p$ minimal.
\item
Choose
$$
W_1 = \{\m_1\} \subset W_2 = \{\m_1,\m_2\} \subset \ldots \subset W
$$
such that $\# W_i = i$ for all $i\ge 1$, 
and $W_i = W_{i-1} \cup \{\m_i \}$
for all $i\ge 2$, and each $\m_i$ is connected in $\Gamma(C)$ to some
$\m_j$ with $j<i$.
\item
For $i=1,2,\ldots$ compute each $\mu(C_{W_i})_p$, and  a generator for it, in succession
by using that $\mu(C_{W_1})_p = \mu(A_\sep/(\m_1 \cap A_\sep))_p$ is given,
and   
for $i>1$ listing all ordered pairs in
$\mu(C_{W_{i-1}})_p \times \mu(A_\sep/(\m_i \cap A_\sep))_p$ and 
testing whether they are in $C_{W_i}$, and using that
$$\mu(C_{W_i})_p = 
C_{W_i} \cap (\mu(C_{W_{i-1}})_p \times \mu(A_\sep/(\m_i \cap A_\sep))_p).$$
This gives a generator of $\mu(C_W)_p$ for each $W$ in the set
$\Omega$ of connected components of $\Gamma(C)$.
Let $\zeta_W\in \prod_{V\in\Omega}\mu(C_V)_p$ be the element
with this generator as its $W$-th
component, and all other components $1$.
\end{enumerate}
\item
View the set $S=\{ \zeta_W : W\in\Omega \}$ in $\mu(C)_p$
via the isomorphism
$\mu(C)_p \cong \prod_{W\in\Omega}\mu(C_W)_p$ of Lemma \ref{Cprodlem},
let $R = \{ \order(\zeta_W)(W\text{-th basis vector})\}$, and
output $\langle S|R\rangle$.
\end{enumerate}
\end{algorithm}

\begin{prop}
Algorithm \ref{BtoCalgor}  
gives correct output and runs in polynomial time.
\end{prop}

\begin{proof}
By Lemma \ref{Cprodlem} we have 
$C \isom \prod_W C_W$. 
Thus, $\mu(C)_p \isom \bigoplus_W \mu(C_W)_p$
so the output of the algorithm is a set of generators for $\mu(C)_p$.
We have 
$$C_{W_i} \subset  
C_{W_{i-1}} \times  C_{\{\m_i\}}, \qquad
C_{\{\m_i\}} = A_\sep/(\m_i \cap A_\sep).$$
Thus,
$$\mu(C_{W_i})_p \subset 
\mu(C_{W_{i-1}})_p \times \mu(A_\sep/(\m_i \cap A_\sep))_p.$$
By Proposition \ref{Wconnprop}, the group 
$\mu(C_{W_i})_p$ injects into each factor,
and each factor is cyclic of prime power order.
Each factor has size polynomial in the size of the algorithm's
inputs (given an order of rank $n$ and an
element of order $p^k$, we have $\varphi(p^k) \le n$ by Lemma \ref{muApremark},
so $p^k \le 2n$). 
By 
Proposition \ref{Wconnprop}(ii) the natural map
$\mu(C_{W_i})_p \to \mu(A_\sep/(\m_1 \cap A_\sep))_p$
is injective,
for all $i$.
As $i$ gets larger, the groups $\mu(C_{W_i})_p$ get smaller or stay the same.
Thus one can list all 
ordered pairs, and then
efficiently test whether they are in $C_{W_i}$. 
It follows from the above that the algorithm runs in polynomial time.

The presentation $\langle S|R\rangle$ is efficient
by Algorithm \ref{genlprinalgor2a} and
Proposition \ref{Balgorrrr3},
since $\mu(C)_p \subset \mu(B)_p$.
\end{proof}

\begin{rem}
A more intelligent algorithm for step (iii)(d) is to use  
that each $\mu(C_{W_i})_p$ is cyclic (by Proposition~\ref{Wconnprop}(iii)),
and that
$\mu(C_{W_i})_p \subset \mu(C_{W_{i-1}})_p$, as follows.
Starting with $i=1$ and incrementing $i$, proceed as follows in place of step (d).
If $\mu(C_{W_{i-1}})_p$ is trivial, stop.
Otherwise, take an element $a_1\in\mu(C_{W_{i-1}})_p$
of order $p$ and for each of the $p-1$ elements 
$b_1\in\mu(A_\sep/(\m_i \cap A_\sep))_p$
of order $p$ test whether $(a_1,b_1)\in C_{W_i}$.
If there are none, stop (the group is trivial for that $W_i$).
If there is such a pair $(a_1,b_1)\in\mu(C_{W_i})$, 
if $\#\mu(C_{W_i})_p  = p$ then stop with $(a_1,b_1)$ as generator,
and otherwise take each $a_2\in\mu(C_{W_{i-1}})_p$
that is a $p$-th root of $a_1$
and for each of the $p$ possible choices of elements 
$b_2\in\mu(A_\sep/(\m_i \cap A_\sep))_p$
that are a $p$-th root of $b_1$, test whether $(a_2,b_2)\in C_{W_i}$.
As soon as such is found,
if $\#\mu(C_{W_i})_p  = p^2$ then stop with $(a_2,b_2)$ as generator,
and otherwise 
continue this process.
Injecting into each component implies one only needs to check ordered pairs
with the same order in each component.
Since $\#\mu(C_{W_i})_p$ divides $\#\mu(C_{W_{i-1}})_p$,
one only needs to go up to elements of order $\#\mu(C_{W_{i-1}})_p$.
The number of trials is $< p\log_p(\#\mu(C_{W_{i-1}})_p)$,
since there are $p$ choices each time, and
there are $\log_p(\#\mu(C_{W_{i-1}})_p)$ steps.
The final $(a_j,b_j)$ found is a generator for $\mu(C_{W_i})_p$.
\end{rem}

\section{Nilpotent ideals in finite rings}
\label{Isect}

Suppose $R$ is a finite commutative ring and $I$ is a nilpotent ideal of $R$.
Algorithm \ref{algwecalledthisd} below solves the discrete logarithm problem in the 
multiplicative group $1+I$, using the finite filtration:
$$
1+I \supset 1+I^2\supset 1+I^4\supset \cdots \supset 1,
$$
the fact that the map $x\mapsto 1+x$ is an isomorphism from
the additive group $I^{2^i}/I^{2^{i+1}}$
to the multiplicative group $(1+I^{2^i})/(1+I^{2^{i+1}})$,
and the fact that the discrete logarithm problem is easy in these additive groups.

We specify a finite commutative ring by giving a presentation for
its additive group,
i.e., a finite set of generators and a finite set of relations,
and for every pair of generators their product is expressed
as a $\Z$-linear combination of the generators.

The following result can be shown using standard methods.

\begin{prop}
\label{wecalledthisb}
There is a deterministic polynomial-time algorithm that, given
a finite commutative ring $R$ and $2$ ideals $I_1$ and $I_2$ of $R$ such that
$I_2 \subset I_1$, computes an efficient presentation of 
the finite abelian group $I_1/I_2$.
\end{prop}

\begin{lem}
\label{wecalledthisa}
Suppose $R$ is a finite commutative ring, $I$ is an ideal of $R$ such that
$I \subset \sqrt{0_R}$, and for each $i\in\Z_{\ge 0}$ the set $B_i$ is a subset of
$I^{2^i}$ such that 
$B_i \cup I^{2^{i+1}}$ generates the additive group $I^{2^i}$. 
Let $\BB=\bigcup_{i\ge 0} B_i$.
Then 
$1+I = \langle 1+b : b\in \BB \rangle$ (as a multiplicative group).
\end{lem}

\begin{proof}
Since $I$ is nilpotent, $1+I^{2^i}$ is a multiplicative group for 
all $i\in\Z_{\ge 0}$.
We have 
$$
I^{2^i}/I^{2^{i+1}} \isom (1+I^{2^i})/(1+I^{2^{i+1}})
$$ 
via $x\mapsto 1+x$.
Since 
$B_i\cup I^{2^{i+1}}$ generates the additive group $I^{2^{i}}$,
we have that $B_i+ I^{2^{i+1}}$ generates  $I^{2^i}/I^{2^{i+1}}$.
If $I^{2^{k+1}}=0$, then $B_k$ generates  $I^{2^k}$ and
$1+B_k$ generates the multiplicative group $1+I^{2^k}$.
It now follows that $1+\BB$ generates $1+I$.
\end{proof}

\begin{algorithm}
\label{algwecalledthisd}
Given
a finite commutative ring $R$, an ideal $I$ of $R$ such that
$I \subset \sqrt{0}$, for each $i\in\Z_{\ge 0}$ a subset $B_i$ of
$I^{2^i}$ such that 
$B_i \cup I^{2^{i+1}}$ generates the additive group $I^{2^i}$, 
with all but finitely many $B_i=\emptyset$,
and $x\in I$, the algorithm computes  
$(m_b)_{b\in \BB} \in \Z^\BB$ with $1+x = \prod_{b\in \BB} (1+b)^{m_b}$,
where $\BB=\bigcup_{i\ge 0} B_i$, as follows.
\end{algorithm}

\begin{enumerate}
\item
Let $x_0=x$. 
For $i=0,1,\ldots$ use 
Proposition \ref{wecalledthisb} 
to find $(m_b)_{b\in B_i}\in\Z^{B_i}$
such that
$$
x_i \equiv \sum_{b\in B_i} m_b b \mod I^{2^{i+1}} \text{ (in $I^{2^i}/I^{2^{i+1}}$).}
$$
Define $x_{i+1}\in I^{2^{i+1}}$ by
$$1+x_{i+1} = (1+x_{i})\prod_{b\in B_i} (1 + b)^{-m_b}.$$
As soon as $x_{i+1}=0$, terminate, 
setting $m_b=0$ for all $b\in B_j$ with $j>i$ and
outputting $(m_b)_{b\in \BB} \in \Z^\BB$.
\end{enumerate}

\begin{prop}
\label{wecalledthisd}
Algorithm \ref{algwecalledthisd}
is a deterministic algorithm that
produces correct outputs in polynomial time.
\end{prop}

\begin{proof}
Since $I$ is a nilpotent ideal, there exists 
$j\in\Z_{\ge 0}$ such that $I^{2^{j}}=0$. Then $x_j=0$ and
the algorithm gives
$$
1+x = 1+x_0 = \prod_{b\in \bigcup_{i < j} B_i} (1+b)^{m_b}
= \prod_{b\in \BB} (1+b)^{m_b}
$$ 
as desired.
\end{proof}

\begin{lem}
\label{wecalledthisc}
There is a deterministic polynomial-time algorithm that, given
a finite commutative ring $R$, an ideal $I$ of $R$ such that
$I \subset \sqrt{0}$, and for each $i\in\Z_{\ge 0}$ a subset $B_i$ of
$I^{2^i}$ such that 
$B_i \cup I^{2^{i+1}}$ generates the additive group $I^{2^i}$, computes  
a $\Z$-basis for the kernel of the map
$\Z^\BB \to 1+I$, $(m_b)_{b\in \BB} \mapsto \prod_b (1+b)^{m_b}$,
where $\BB=\bigcup_{i\ge 0} B_i$.
\end{lem}

\begin{proof}
Let $C_j=\bigcup_{k\ge j}B_j$. 
We proceed by induction on decreasing $j$.
We have $\langle 1+C_j\rangle = 1+I^{2^j}$
(applying Lemma \ref{wecalledthisa} with $I^{2^j}$ in place of $I$).
Assume we already have defining relations for $1+C_j$, 
i.e., we have generators for the kernel of
$\Z^{C_j} \to 1+I^{2^j}$, 
$(m_b)_{b\in C_j} \mapsto \prod_{b\in C_j} (1+ b)^{m_b}$,
and would like to find
defining relations for $1+C_{j-1}$.
Proposition \ref{wecalledthisb}
gives an algorithm for finding
a basis for 
the kernel of
$\Z^{B_{j-1}} \to I^{2^{j-1}}/I^{2^j}$, 
$(n_b)_{b\in B_{j-1}} \mapsto \prod_{b\in B_{j-1}} {n_b}b + I^{2^j}$ 
in polynomial time.
For each defining relation $(n_b)_{b\in B_{j-1}}$ for $B_{j-1}+I^{2^j}$ we have
$\sum_{b\in B_{j-1}} {n_b}b\equiv 0$ mod $I^{2^j}$ so
$\prod_{b\in B_{j-1}} (1+b)^{n_b} \equiv 1$ mod $(1+I^{2^j})$. 
Algorithm \ref{algwecalledthisd} gives a polynomial-time algorithm to find
$(m_{b'})_{b'\in C_j}\in\Z^{C_j}$ such that 
$\prod_{b\in B_{j-1}} (1+b)^{n_b} = \prod_{b'\in C_j} (1+b')^{m_{b'}}\in 1+I^{2^j}$.
Then $((n_b)_{b\in B_{j-1}},(-m_{b'})_{b'\in C_j})$ 
is in the kernel of the map $\Z^{C_{j-1}} \to 1+I^{2^{j-1}}$,
and these relations along with the defining relations for $1+C_j$
form a set of defining relations for $1+C_{j-1}$.
\end{proof}

\begin{thm}
\label{1plusIpresentthm}
There is a deterministic polynomial-time algorithm
that, given a finite commutative ring and an ideal $I$ of $R$ such that
$I \subset \sqrt{0}$, 
produces 
an efficient presentation $\langle 1+\BB|\RR \rangle$ for $1+I$. 
\end{thm}

\begin{proof}
Apply the algorithm in 
Proposition \ref{wecalledthisb} to obtain 
for each $i\in\Z_{\ge 0}$ a set $B_i \subset I^{2^i}$
such that 
$B_i \cup I^{2^{i+1}}$ generates the additive group $I^{2^i}$.
Since $I$ is nilpotent, we can take
$B_i=\emptyset$ for all but finitely many $i$. 
By Lemma \ref{wecalledthisa} 
the set 
$\BB=\bigcup_{i\ge 0} B_i$ has
the property that $1+\BB$ generates $1+I$.
Defining relations $\RR$ are given by Lemma \ref{wecalledthisc},
and part (c) of Definition \ref{effpresdefn} holds by Proposition \ref{wecalledthisd}.
\end{proof}

Theorem \ref{muAthm4} now follows from Theorem \ref{1plusIpresentthm} 
and Algorithm \ref{genlprinalgor2a}.

\begin{rem}
Suppose $R$ is a finite commutative ring, $I\subset R$ is a nilpotent ideal,
and $R'$ is a subring of $R$.
Let $I'=I\cap R'$.
The algorithm in Theorem \ref{1plusIpresentthm}
gives efficient presentations for the multiplicative groups $1+I$
and $1+I'$. 
We can apply Algorithm \ref{genlprinalgor4} 
with $G = 1+I \subset R^\ast$, and
$T'$ a set of generators for $1+I'$,
and $T$ a set of generators for some subgroup of $1+I$.
In the next section we will apply this to our setting.
\end{rem}

\begin{ex}
Let $R=\Z/p^2\Z$ and $I=\sqrt{0_R} = p\Z/p^2\Z$. 
Then $I^2=0$, and $1+I$ is the order $p$ subgroup of
$(\Z/p^2\Z)^\ast \cong \Z/p\Z \times \Z/(p-1)\Z$.
The map $1+I \isom \Z/p\Z$, $1+x\mapsto x/p$ is a group isomorphism,
so the discrete logarithm problem is easy in $1+I$.
\end{ex}

\begin{ex}
Let $R=\Z/p^4\Z$ and $I=\sqrt{0_R} = p\Z/p^4\Z$. 
Then $I^4=0$. Here,
the map $1+I \isom \Z/p^3\Z$, $1+x\mapsto x/p$ is not a group homomorphism.
The discrete logarithm problem is easy in $1+I$ not because it is
(isomorphic to) an additive group, but because there is a filtration
of additive groups, namely,
$
(1+I)/(1+I^2) \cong I/I^2$ and $(1+I^2)/(1+I^4) \cong I^2/I^4 = I^2.
$
\end{ex}

\section{From $\mu(C)_p$ to $\mu(A)_p$}

Let $A$ be an order and let $p$ be a prime.
Recall $C$ from Definition \ref{Cdefn} and
let
$$
\ff = \{ x\in C : xC \subset A_\sep \},
$$
which is the largest ideal of $C$ that is contained in $A$.
We shall see that $C/\ff$ is a finite ring, and it has 
$A_\sep/\ff$ as a subring.
Suppose we are given a set $M \subset C^\ast$ such that 
$\mu(C)_p=\langle M\rangle$. Let
$$
I = \sum_{\zeta\in M} (\zeta - 1)(C/\ff), \qquad I'=I\cap (A_\sep/\ff).
$$
Define 
$$
g_1 : \Z^M \onto \mu(C)_p, \qquad
(a_\zeta)_{\zeta\in M} \mapsto \prod_{\zeta\in M} \zeta^{a_\zeta},
$$
let $g_2 : \mu(C)_p \to 1+I$ be the natural map $\zeta\mapsto\zeta + \ff$,
let ${\hat{g}} : \mu(C)_p  \to (1+I)/(1+I')$ denote the composition of $g_2$ with
the quotient map,
define $g : \Z^M \to 1+I$ by $g = g_2\circ g_1$,
and define
\begin{equation}
\label{psidef}
\psi : \Z^M \to (1+I)/(1+I') \quad \text{  by } \quad \psi = {\hat{g}}\circ g_1.
\end{equation}

\begin{prop}
\label{Ipsietc}
With notation as above,
\begin{enumerate}
\item 
$I$ is a nilpotent ideal of $C/\ff$, i.e., $I \subset \sqrt{0_{C/\ff}}$;
\item 
$I'$ is a nilpotent ideal of $A_\sep/\ff$;
\item 
$C/\ff$ is a finite ring of $p$-power order,
\item 
$\mu(A)_p$ is the kernel of the map ${\hat{g}}$;
\item 
$\mu(A)_p$ is 
the image of 
$\ker(\psi)$ under the map $g_1$.
\end{enumerate}
\end{prop}

\begin{proof}
Since $C/A$ is killed by $p^r$ for some $r\in\Z_{\ge 0}$, we have
$p^r\in\ff$, so $p \in \sqrt{0_{C/\ff}}$, so $p$ is in every
prime ideal of $C/\ff$. 
Suppose $\zeta\in \mu(C)_p$.
Then the image of $\zeta$ in every field of characteristic $p$ is $1$.
Thus, $\zeta - 1$ is in every prime ideal of $C/\ff$, so 
$\zeta - 1 \in \sqrt{0_{C/\ff}}$.
By the definition of $I$ we have $I \subset \sqrt{0_{C/\ff}}$, and (i) and (ii) follow.

Since $p^r\in\ff$ we have $p^rC \subset \ff$, so $C/\ff$ is a quotient
of $C/p^rC$, which is a finite ring of $p$-power order. This gives (iii).

Part (iv) follows directly from the definitions, and then (v)
follows from (iv).
\end{proof}

\begin{algorithm}
\label{I1coralg}
The algorithm takes as input 
an order $A$, a prime $p$, and a finite set of generators $M$ 
for $\mu(C)_p$, 
and computes a finite set of generators for $\mu(A)_p$.
\end{algorithm}

\begin{enumerate}
\item
Compute the finite abelian group $C/A_\sep$ and
$$\Hom(C,C/A_\sep) \cong 
(C/A_\sep) \oplus (C/A_\sep) \oplus \cdots \oplus (C/A_\sep)$$
(with $\rk_\Z(C)$ summands $C/A_\sep$),
and compute $\ff$ as the kernel of the group
homomorphism $A_\sep \to \Hom(C,C/A_\sep)$
sending $x\in A_\sep$ to the map $y \mapsto xy+A_\sep$.
Next compute the finite rings $A_\sep/\ff \subset C/\ff$.
This entire step can be done using standard
algorithms for finitely generated abelian groups.
\item
Apply the algorithm in Theorem \ref{1plusIpresentthm}
with $R=C/\ff$ and the $I$ of this section to obtain an efficient
presentation for $1+I$.
\item
Apply the algorithm in Theorem \ref{1plusIpresentthm}
with $R=A_\sep/\ff$ and $I'$ in place of $I$ to obtain a
finite set $T'$ of generators for $1+I'$.
\item
Apply Algorithm \ref{genlprinalgor4}
with $G = 1+I$, the efficient presentation from step (ii),
$T=M$, and $T'$ from step (iii) to obtain a finite set of 
generators $S'$ for
$\ker(\Z^T \to G/\langle T'\rangle)$.
\item
Take the image of $S'$ under the map $g_1 : \Z^M \to \mu(C)_p$.
\end{enumerate}

\begin{thm}
\label{I1coralgworks}
Algorithm \ref{I1coralg} produces correct output and runs in polynomial time.
\end{thm}

\begin{proof}
Since $C/\ff$ and $A_\sep/\ff$ are finite commutative rings, and $I$ 
and $I'$ are nilpotent, 
Theorem \ref{1plusIpresentthm} 
is applicable in steps (ii) and (iii).
The map $\Z^M = \Z^T \to G/\langle T'\rangle = (1+I)/(1+I')$
in step (iv) is our map $\psi$ from \eqref{psidef}.
By Proposition \ref{Ipsietc}(v), step (v) produces generators for $\mu(A)_p$.
\end{proof}

\section{Finding roots of unity}
\label{roualgsect}

\begin{algorithm}
\label{rootsofunityalg}
Given an order $A$, the algorithm outputs a finite set of generators for $\mu(A)$.
\end{algorithm}

\begin{enumerate}
\item
Use Algorithm \ref{findmalgor} to compute $E_\sep$, 
all $\m\in\Spec(E)$, the fields $E/\m$, and the natural maps $E \to E/\m$.
\item
Apply Algorithm \ref{Asepalgor} to 
compute
$A_\sep = A \cap E_\sep$. 
\item
Apply Algorithm \ref{Balgorrrr} to compute
for each $\m\in\Spec(E)$ 
 the subring 
$A_\sep/(\m\cap A_\sep)$ of $E_\sep/\m$. 
\item
Apply the algorithm in Proposition \ref{Balgorrrr2} to compute,
for each $\m\in \Spec(E)$, 
a generator $\theta_\m$ for $\mu(A_\sep/(\m\cap A_\sep))$, 
its order,
the prime factorization of its order,
and for each prime $p$ dividing its order a generator
$\theta_{\m,p}$ of $\mu(A_\sep/(\m\cap A_\sep))_p$.
\item
For each prime  $p$ dividing the order of at least one of the
groups $\mu(A_\sep/(\m\cap A_\sep))$, do the following:
\begin{enumerate}
\item
Use the image algorithm
in \S 14 of \cite{HWLMSRI} to compute 
a $\Z$-basis for $C = A_\sep[{1/p}]\cap B$
(as discussed in \S\ref{BtoCsect} above, just before Proposition \ref{nonpcptprop}).
\item
Apply Algorithm \ref{BtoCalgor} to compute 
an efficient presentation for $\mu(C)_p$.
\item
Apply Algorithm \ref{I1coralg} to compute generators for $\mu(A)_p$.
\end{enumerate}
\item 
Generators for these groups $\mu(A)_p$ form a set of generators for $\mu(A)$.
\end{enumerate}

That Algorithm \ref{rootsofunityalg} produces correct output
and runs in polynomial time follows immediately.
We can now obtain a deterministic polynomial-time algorithm
that, given an order $A$,
determines an efficient presentation for $\mu(A)$.

\begin{algorithm}
\label{rootsofunityrelsalg}
The algorithm takes an order $A$ and produces
an efficient presentation for $\mu(A)$.
\end{algorithm}

\begin{enumerate}
\item
Apply the algorithm in
Proposition \ref{Balgorrrr3} to obtain an efficient presentation $\langle S|R\rangle$ 
for $\mu(B)$.
\item
Apply Algorithm \ref{rootsofunityalg} to obtain a 
finite set of generators 
for $\mu(A)$.
\item
Apply Algorithm \ref{genlprinalgor2a} with $G=\mu(B)$ to obtain 
an efficient presentation for $\mu(A)$.
\end{enumerate}

\begin{ex}
Let $A = \Z[X]/(X^{4}-1)$. Then with $p=2$:
$$
B = C =
\Z[X]/(X-1)  \times \Z[X]/(X+1) \times \Z[X]/(X^{2}+1) \cong \Z \times \Z \times \Z[\i],
$$
and $(C:A)=8$.
We identify $X$ with $(1,-1,\i) \in \Z \times \Z \times \Z[\i]$.
Then
$$
\mu(A)_2 = \mu(A) \subset \mu(B) = \mu(C)_2 = \langle (-1,1,1), (1,-1,1),(1,1,\i)\rangle. 
$$
We have 
$$
\ff = 4\Z \times 4\Z \times 2\Z[\i]
$$
of index 64 in $C$, and
$$
C/\ff = \Z/4\Z \times \Z/4\Z \times \Z[\i]/2\Z[\i] = 
\Z/4\Z \times \Z/4\Z \times \F_2[\varepsilon]
$$
with $\varepsilon = 1+\i$.
The index 8 subring of $C/\ff$ generated by $(1,-1,1+\varepsilon)$ is $A/\ff$.
Alternatively,
$$
A/\ff = (\Z/4\Z)[Y]/(2Y,Y^2)
$$
where $Y = X-1 = (0,2,\varepsilon)\in A/\ff$.
With $M = \{ (-1,1,1), (1,-1,1),(1,1,\i)\}$ we have
$$
I = (2\Z/4\Z) \times (2\Z/4\Z) \times (\varepsilon\F_2[\varepsilon]) = \sqrt{0_{C/\ff}},
$$
$I^2=0$,  and
$$
I' = I \cap (A/\ff) = \sqrt{0_{A/\ff}} = 
\{0,2,Y,Y+2 \}.
$$
With $\psi$ as in \eqref{psidef},
we have $\psi(a,b,c) = a+b+c+2\Z\in\Z/2\Z$ and
$$
\ker(\psi) = \{ (a,b,c)\in\Z^M : a+b+c \text{ is even}\} = 
\Z\cdot(2,0,0) + \Z\cdot(1,1,0) + \Z\cdot(1,0,1).
$$
Algorithm \ref{rootsofunityalg} outputs
$$
\mu(A) = \mu(A)_2 = \langle -X^2 \rangle \times \langle -X^3 \rangle  = \langle X,-1 \rangle 
\cong \Z/2\Z \times \Z/4\Z.
$$
\end{ex}

\begin{ex}
Let $A = \Z[X]/(X^{12}-1)$. Then
$$
E = \Q[X]/(X^{12}-1) \cong 
\Q \times \Q \times \Q(\zeta_3) \times \Q(\i) \times \Q(\zeta_3) \times \Q(\zeta_{12})
$$
and
\begin{eqnarray*}
& B = 
& \Z[X]/(X-1)   \, \, \times \, \, \Z[X]/(X+1)  \, \, \times \, \, \Z[X]/(X^{2}+X+1) \\
&& \times  \, \, \Z[X]/(X^{2}+1)  \, \, \times  \, \, \Z[X]/(X^{2}-X+1) \, \, \times \, \, \Z[X]/(X^4-X^2+1) \hookrightarrow E.
\end{eqnarray*}
We have for the discriminants of the orders: 
$$
|\Delta_B| =  1\cdot 1\cdot 3\cdot 4\cdot 3\cdot 12^2, \qquad
|\Delta_A| = 12^{12},
$$
so
$$
\#(B/A) = \sqrt{|\Delta_A|/|\Delta_B|} = 2^9\cdot 3^4.
$$
Thus if $p=2$ then $(C:A)=2^9$, while if $p=3$ then $(C:A) = 3^4$.
The graph $\Gamma(B)$ consists of  6 vertices with no edges.
With the numbers $n(A,\m,\nn)$ on the edges,
the graph $\Gamma(A)$ is:
$$
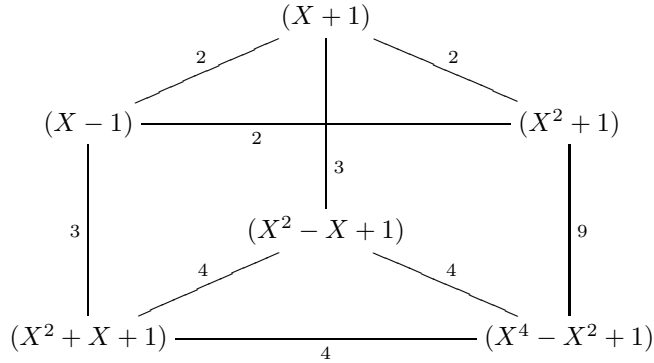
\xymatrix{
 & (X+1) 
 \ar@{-}_{2}[ld] \ar@{-}^{2}[rd] \ar@{-}^(.7){3}[dd]  \\
(X-1) \ar@{-}_(.35){2}[rr] \ar@{-}_-{3}[dd] & &  (X^2+1)  \ar@{-}^-{9}[dd]  \\
& (X^2-X+1) \ar@{-}_-{4}[ld]  \ar@{-}^-{4}[rd]   \\
(X^2+X+1)  \ar@{-}_-{4}[rr]  && (X^4-X^2+1) 
}
$$

Suppose $p=2$. Then the graph $\Gamma(C)$ 
is:
$$
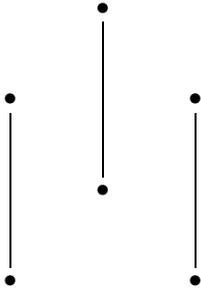
\xymatrix{
   & \bullet  \ar@{-}[dd]  \\
\bullet   \ar@{-}[dd] && \bullet \ar@{-}[dd] \\
& \bullet \\
\bullet   && \bullet  
}
$$

We have $\mu(C)_2 = \prod\mu(C_W)_2$ with the product
running over the 3 connected components $W$.
The 
left 2 $W$'s give $\mu(C_W)_2 = \{\pm 1\}$, while
the remaining one gives $\mu(C_W)_2 = \langle -X^3\rangle$.
This gives
$
 -X^3, -1 \in\mu(A)_2.
$

Suppose $p=3$. Then the graph $\Gamma(C)$ 
is:
$$
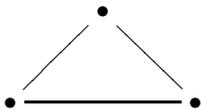
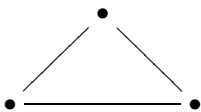
\xymatrix{
 & \bullet \ar@{-}[ld] \ar@{-}[rd]   \\
\bullet \ar@{-}[rr]  & &  \bullet \\
& \bullet \ar@{-}[ld]  \ar@{-}[rd]   \\
\bullet  \ar@{-}[rr]  && \bullet
}
$$

We have $\mu(C)_3 = \prod\mu(C_W)_3$ with the product
running over the 2 connected components $W$.
The  
top $W$ has $\mu(C_W)_3 = \{ 1\}$, while
for the  
bottom $W$ one has that $\mu(C_W)_3$ is generated by the image of $X^4$,
and this gives
$
X^4 \in \mu(A)_3.
$

Continuing the algorithm by hand is more complicated than
in the previous example.
However, we note that here $A$ is the order $\Z\langle G\rangle$
defined in \cite{LwS} with 
$G = \langle -1\rangle \times \langle X\rangle \cong \Z/2\Z \times \Z/12\Z$,
and it follows from Remark 16.3 of \cite{LwS} that $\mu(A) = G
= \langle -1\rangle \times \langle X\rangle$.
\end{ex}


\begin{thebibliography}{99}
	\bibitem{AtiyahMcD}
M.\ F.\ Atiyah and I.\ G.\ Macdonald,  
Introduction to commutative algebra, 
Addison-Wesley Publishing Co., Reading, MA, 1969.
	\bibitem{HopcroftTarjan}
J.\ Hopcroft and R.\ Tarjan,
{\em Algorithm 447: efficient algorithms for graph manipulation},
Communications of the ACM,
{\bf 16}, no.\ 6 
(1973) 
372--378. 
	\bibitem{LangAlg}
S.\ Lang, Algebra, Third edition, 
Graduate Texts in Mathematics {\bf 211}, Springer-Verlag, New York, 2002.
	\bibitem{Arjen}
A.\ K.\ Lenstra, 
{\em Factoring polynomials over algebraic number fields},
in Computer algebra (London, 1983), 
Lect.\ Notes in Comp.\ Sci.\ {\bf 162}, Springer, Berlin, 1983, 245--254. 
	\bibitem{HWLMSRI}
H.\ W.\ Lenstra, Jr., {\em  Lattices},
in Algorithmic number theory: lattices, number fields, curves and cryptography, 
Math.\ Sci.\ Res.\ Inst.\ Publ.\ {\bf 44}, Cambridge Univ.\ Press, Cambridge, 2008, 127--181, \url{http://library.msri.org/books/Book44/files/06hwl.pdf}.
	\bibitem{LenSil}
H.\ W.\ Lenstra, Jr.\ and A.\ Silverberg,
{\em Revisiting the Gentry-Szydlo Algorithm}, 
in Advances in Cryptology---CRYPTO 2014,  
Lect.\ Notes in Comp.\ Sci.\ {\bf  8616}, Springer, Berlin, 2014, 280--296.
	\bibitem{LwS}
H.\ W.\ Lenstra, Jr.\ and A.\ Silverberg,
{\em Lattices with symmetry}, to appear in Journal of Cryptology,
\url{https://eprint.iacr.org/2014/1026}.
	\bibitem{Qalgs}
H.\ W.\ Lenstra, Jr.\ and A.\ Silverberg,
{\em Algorithms for commutative algebras over the rational numbers},
\url{http://arxiv.org/abs/1509.08843}.
\end{thebibliography}
\end{document}